\documentclass[11pt]{article}

\usepackage{amsmath,amssymb,latexsym,amsthm}
\usepackage{tikz}
\usepackage[T1]{fontenc}
\usepackage{cite}

\usepackage{hyperref}
\hypersetup{colorlinks=true}
\hypersetup{colorlinks=true, linkcolor=blue, citecolor=blue,urlcolor=blue}

\newtheorem{theorem}{Theorem}[section]
\newtheorem{definition}[theorem]{Definition}
\newtheorem{lemma}[theorem]{Lemma}
\newtheorem{proposition}[theorem]{Proposition}
\newtheorem{observation}[theorem]{Observation}

\newtheorem{corollary}[theorem]{Corollary}

\newcommand{\pad}[2]{\Phi^{#1}_{#2}}
\def\cp{\,\square\,}
\DeclareMathOperator{\aut}{Aut}
\DeclareMathOperator{\diam}{diam}

\usepackage[normalem]{ulem}

\title{Weighted Padovan graphs}

\author{
Vesna Ir\v si\v c Chenoweth $^{1,2,}$\thanks{Email: \texttt{vesna.irsic@fmf.uni-lj.si}}
\and
Sandi Klav\v zar $^{1,2,3,}$\thanks{Email: \texttt{sandi.klavzar@fmf.uni-lj.si}}
\and 
Gregor Rus $^{1,2,}$\thanks{Email: \texttt{gregor.rus@fmf.uni-lj.si}}
\and
Elif Tan $^{4,}$\thanks{Email: \texttt{etan@ankara.edu.tr}}
}

\begin{document}

\maketitle

\begin{center}
$^1$ Faculty of Mathematics and Physics,  University of Ljubljana, Slovenia\\
\medskip

$^2$ Institute of Mathematics, Physics and Mechanics, Ljubljana, Slovenia\\
\medskip

$^3$ Faculty of Natural Sciences and Mathematics,  University of Maribor, Slovenia\\
\medskip

$^4$ Department of Mathematics, Ankara University, Ankara, Türkiye  \\
\end{center}

\begin{abstract}
Weighted Padovan graphs $\Phi^{n}_{k}$, $n \geq 1$, $\lfloor \frac{n}{2} \rfloor \leq k \leq \lfloor \frac{2n-2}{3} \rfloor$, are introduced as the graphs whose vertices are all Padovan words of length $n$ with $k$ $1$s, two vertices being adjacent if one can be obtained from the other by replacing exactly one $01$ with a $10$. By definition, $\sum_k |V(\Phi^{n}_{k})|=P_{n+2}$, where $P_n$ is the $n$th Padovan number. 
Two families of graphs isomorphic to weighted Padovan graphs are presented. The order, the size, the degree, the diameter, the cube polynomial, and the automorphism group of weighted Padovan graphs are determined. It is also proved that they are median graphs.
\end{abstract}


\medskip
\noindent
{\bf AMS Subj.\ Class.\ (2020):} 05C75, 05C12, 05C30, 11B39

\maketitle

\section{Introduction}

The {\em Fibonacci sequence} is one of the most famous sequences in mathematics. The $n$th Fibonacci number $F_{n}$ is defined by $F_{n}=F_{n-1}+F_{n-2}$, $n\geq 2$, with initial values $F_{0}=0$ and $F_{1}=1$. Its analogous sequence, {\em Lucas sequence}, has the same recurrence relation but begins with initial values $L_{0}=2$ and $L_{1}=1$. Similarly, the {\em Pell sequence} has the same initial values as the Fibonacci sequence, but now the recurrence reads as the sum of twice the previous term plus the pre-previous term. Fibonacci numbers and their generalizations have many interesting properties and many applications in science and art, see~\cite{koshy-2019}. We also refer to the book~\cite{koshy-2014} for Pell and Pell-Lucas numbers and their applications.

The Fibonacci sequence and the Lucas sequence inspired the investigation of different interesting families of graphs such as Fibonacci and Lucas cubes~\cite{egecioglu-2021b, hsu-1993, klavzar-2023, mollard-2021, munarini-2001}, Pell graphs~\cite{munarini-2019}, generalized Pell graphs~\cite{irsic-2023}, metallic cubes~\cite{doslic-2024}, Fibonacci and Lucas $p$-cubes~\cite{wei-2022}, and Fibonacci-run graphs~\cite{egecioglu-irsic-2021}, to list just some of them. The state of research on Fibonacci cubes and related topics up to 2013 is summarised in the survey paper~\cite{klavzar-2013}, while for the state of the art results on Fibonacci and related cubes see the 2023 book~\cite{egecioglu-2023}.

The Padovan numbers, which are named after the architect 
Richard Padovan, see~\cite{padovan}, are defined by the third order recurrence relation%
\begin{equation*}
P_{n}=P_{n-2}+P_{n-3},\ n\geq 3,
\end{equation*}%
with initial values $P_{0}=1$, $P_{1}=P_{2}=0$. In the Online Encyclopedia of Integer Sequences, the Padovan sequence appears as~\cite[A000931]{OEIS}. The first few terms are
\begin{equation*}
1,0,0,1,0,1,1,1,2,2,3,4,5,7,9,12,16,21,28,37,49,65,86,114,\ldots
\end{equation*}%
The associated generating function is 
\begin{equation*}
\sum_{n\geq 0}P_{n}x^{n}=\frac{1-x^2}{1-x^{2}-x^{3}}.
\end{equation*}%

Recently, Lee and Kim~\cite{lee-2023-1} introduced the {\em Padovan cubes} by using only the odd terms $P_1, P_3, P_5, \ldots$ of the Padovan sequence. A motivation for their definition is that every positive integer can be represented uniquely as the sum of one or more odd terms of the Padovan sequence such that the sum does not include any three consecutive odd terms. In~\cite{lee-2023-2} the investigation of the Padovan cubes continued by investigating their cube polynomials. 

The main impetus for the investigation in this paper is to consider all the terms of the Padovan sequence to construct a respective family of graphs. This is done formally in the next section by introducing the weighted Padovan graphs $\pad{n}{k}$. In the same section two isomorphic families of graphs are presented. In Section~\ref{sec:order-size-degree}, we determine  the order, the size and the degree of weighted Padovan graphs. In the subsequent section we investigate their metric properties, and in particular prove that they are median graphs.  In Section~\ref{sec:cube-poly}, the cube polynomial of weighted Padovan graphs is determined, as well as its generating function. In the last section we find all symmetries of the studied family of graphs.

\section{Weighted Padovan graphs and two isomorphic families}
\label{sec:three families}

In this section we introduce the weighted Padovan graphs and two isomorphic families of graphs which will both be useful for proving properties of weighted Padovan graphs in the rest of the paper. The first of the two families has a word representation, while the second one is defined by integer partitions.

If $A$ is an alphabet, then a {\em word} over $A$ is a sequence of letters from $A$. When $A=\{0,1\}$, we speak of a {\em binary word}. By a {\em subword} of a word we mean a subsequence of consecutive letters of the word.

\begin{definition}
A binary word is \emph{Padovan}, if it
\begin{itemize}
\item starts and ends with $0$, 
\item contains no subword $00$, and
\item contains no subword $111$.
\end{itemize}
By $\mathcal{P}_{n}$ we denote the set of Padovan words of length $n\ge 1$. 
\end{definition}
Note that $|\mathcal{P}_{0}| = |\mathcal{P}_{1}| = |\mathcal{P}_{2}| = 1$. Moreover, as noted by Yifan Xie in~\cite[Sequence A000931]{OEIS}, if $n\ge 3$, then we have $$|\mathcal{P}_{n}|=P_{n+2}\,.$$

Since Padovan words contain no subword $00$, the minimum number of $1$s in a Padovan word of length $n$ is $\left \lfloor \frac{n}{2} \right \rfloor$. This is attained, for example, by the binary word $01\ldots010$ if $n$ is odd, and by $01\ldots0110$ if $n$ is even. Similarly, since Padovan words contain no $111$, the maximum number of $1$s in a Padovan word of length $n$ is $\left \lfloor \frac{2n - 2}{3} \right \rfloor$. This is attained, for example, by $011\ldots011010$ if $n \bmod 3  = 0$, by $011\ldots 0110$ if $n \bmod 3 = 1$, and by $011 \ldots 01101010$ if $n \bmod 3 = 2$. Thus there are 
$$\left \lfloor \frac{2n - 2}{3} \right \rfloor - \left \lfloor \frac{n}{2} \right \rfloor + 1 = \left \lfloor \frac{n+1}{2} \right \rfloor - \left \lfloor \frac{n+1}{3} \right \rfloor$$ 
different possibilities for the number of $1$s in a Padovan word of length $n$. 

Our key definition now reads as follows.

\begin{definition}
\label{def:padovan}
    The \emph{Padovan graph $\pad{n}{k}$ of length $n$ and weight $k$}, $n \geq 1$, $\left \lfloor \frac{n}{2} \right \rfloor \leq k \leq \left \lfloor \frac{2n-2}{3} \right \rfloor$, is the graph whose vertices are all Padovan words of length $n$ with $k$ $1$s, two vertices being adjacent if one can be obtained from the other by replacing exactly one subword $01$ with a $10$. The family of these graphs will be called {\em weighted Padovan graphs}.  
\end{definition}

\begin{observation}
    \label{obs:padovan-vertices}
    If $n \geq 1$, then $$\sum_{k = \left \lfloor \frac{n}{2} \right \rfloor}^{\left \lfloor \frac{2n-2}{3} \right \rfloor} |V(\pad{n}{k})| = P_{n+2}\,.$$
\end{observation}

The switching adjacency rule ``$01$ to $10$'' in the definition of a Padovan graph clearly preserves the number of $1$s. In order that the graphs considered are connected, this is the reason the parameter $k$ counting the number of $1$s is present. For the graphs to be connected it is thus necessary to use parameter $k$.

In Table~\ref{tab:small-Padovan} the Padovan graphs $\pad{n}{k}$ are listed for $n\le 10$ and all respective $k$s. 

\begin{table}[ht!]
    \centering
    \begin{tabular}{|c|c|c|c|}
    \hline
         $n$ & $\left \lfloor \frac{n}{2} \right \rfloor$ & $\left \lfloor \frac{2n-2}{3} \right \rfloor$ & $\{\pad{n}{k}:\ \left \lfloor \frac{n}{2} \right \rfloor \le k\le \left \lfloor \frac{2n-2}{3} \right \rfloor\}$  \\
         \hline\hline
1 & 0 & 0 & $\{K_1\}$ \\ \hline
2 & 1 & 0 & $\emptyset$ \\ \hline 
3 & 1 & 1 & $\{K_1\}$ \\ \hline
4 & 2 & 2 & $\{K_1\}$ \\ \hline
5 & 2 & 2 & $\{K_1\}$ \\ \hline
6 & 3 & 3 & $\{K_2\}$ \\ \hline
7 & 3 & 4 & $\{K_1\}$ \\ \hline
8 & 4 & 5 & $\{P_3\}$ \\ \hline
9 & 4 & 5 & $\{K_1, P_3\}$ \\ \hline
10 & 5 & 6 & $\{K_1, P_4\}$ \\ \hline 
    \end{tabular}
 \vspace*{5mm}
    \caption{The list of Padovan graphs with $n\le 10$. Here $P_3$ and $P_4$ denote the path graphs on three and four vertices, respectively.}
    \label{tab:small-Padovan}
\end{table}

For $n \geq 11$, the structure of the graphs becomes more interesting. For example, if $n=11$, possible weights $k$ are $5$ and $6$, and the obtained graphs are drawn in Figure~\ref{fig:padovan-11}. Another example is shown in Figure~\ref{fig:padovan-15}, where $n = 15$, thus $k \in \{7,8,9\}$.

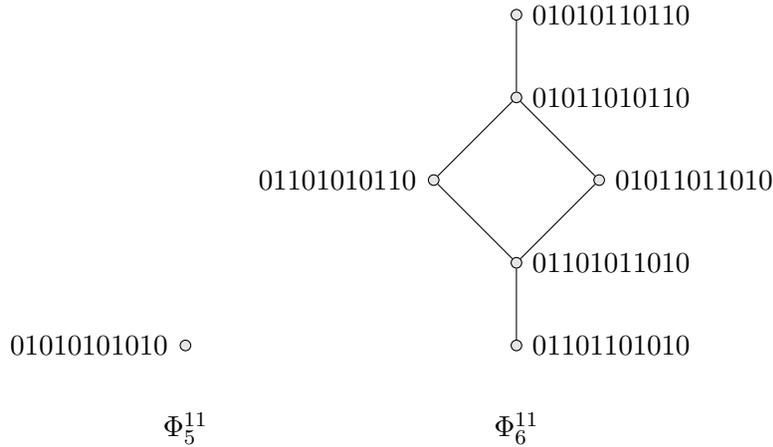
\begin{figure}[ht!]
    \centering
    \begin{tikzpicture}
    [scale=1.1,
    vert/.style={circle, draw, fill=black!10, inner sep=0pt, minimum width=4pt}, 
    double/.style={circle, draw, fill=black, inner sep=0pt, minimum width=4pt}, 
    central/.style={circle, draw, fill=black, inner sep=0pt, minimum width=4pt},
    ]

    \node[vert, label=right:$01101101010$] (0) at (0,0) {};
    \node[vert, label=right:$01101011010$] (1) at (0,1) {};
    \node[vert, label=left:$01101010110$] (2) at (-1,2) {};
    \node[vert, label=right:$01011011010$] (3) at (1,2) {};
    \node[vert, label=right:$01011010110$] (4) at (0,3) {};
    \node[vert, label=right:$01010110110$] (5) at (0,4) {};    
 
    \draw (0) -- (1) -- (2) -- (4) -- (5);
    \draw (1) -- (3) -- (4);

    \node (l1) at (0,-1) {$\pad{11}{6}$};

    \node[vert, label=left:$01010101010$] (6) at (-4,0) {};

    \node (l2) at (-4,-1) {$\pad{11}{5}$};
    
    \end{tikzpicture}
    \caption{Both weighted Padovan graphs for $n = 11$.}
    \label{fig:padovan-11}
\end{figure}

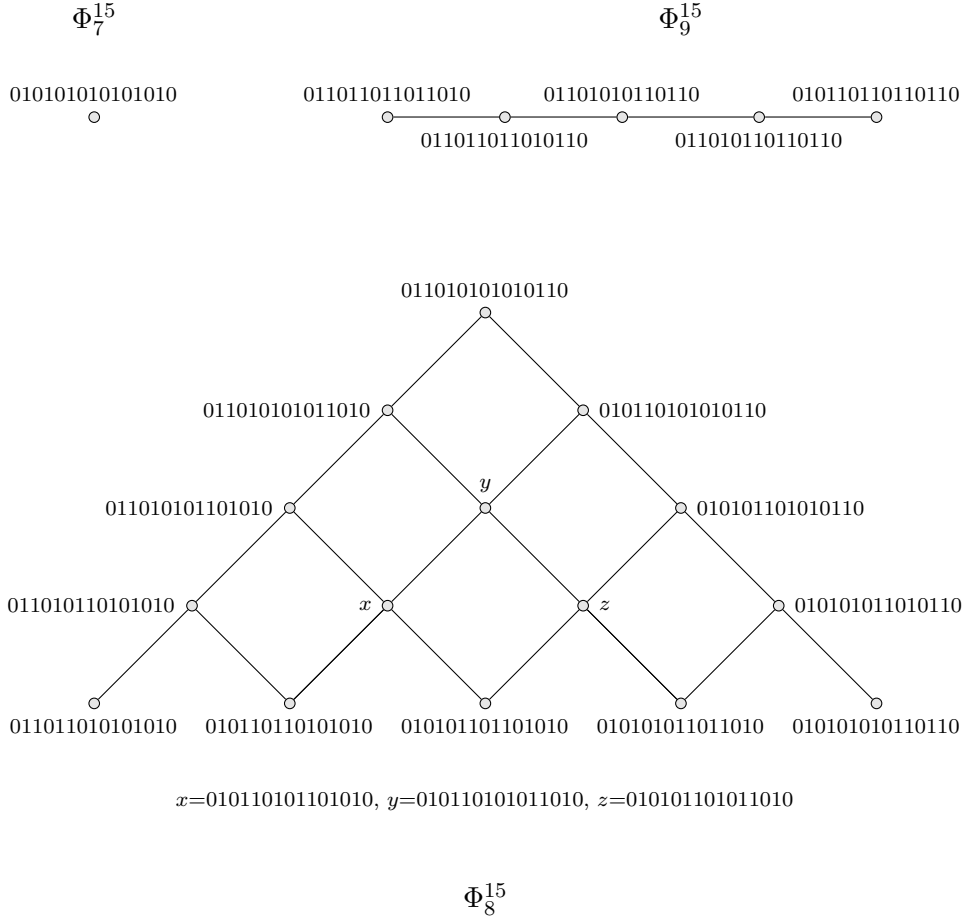
\begin{figure}[ht!]
    \centering
    \begin{tikzpicture}
    [scale=1.3,
    vert/.style={circle, draw, fill=black!10, inner sep=0pt, minimum width=4pt}, 
    double/.style={circle, draw, fill=black, inner sep=0pt, minimum width=4pt}, 
    central/.style={circle, draw, fill=black, inner sep=0pt, minimum width=4pt},
    ]

    \node[vert, label=below:$\scriptstyle{011011010101010}$] (0) at (0,0) {};
    \node[vert, label=left:$\scriptstyle{011010110101010}$] (1) at (1,1) {};
    \node[vert, label=left:$\scriptstyle{011010101101010}$] (2) at (2,2) {};
    \node[vert, label=left:$\scriptstyle{011010101011010}$] (3) at (3,3) {};
    \node[vert, label=above:$\scriptstyle{011010101010110}$] (4) at (4,4) {};
    
    \node[vert, label=below:$\scriptstyle{010110110101010}$] (5) at (2,0) {};
    \node[vert, label=left:$\scriptstyle{x}$] (6) at (3,1) {};
    \node[vert, label=above:$\scriptstyle{y}$] (7) at (4,2) {};
    \node[vert, label=right:$\scriptstyle{010110101010110}$] (8) at (5,3) {};

    \node[vert, label=below:$\scriptstyle{010101101101010}$] (9) at (4,0) {};
    \node[vert, label=right:$\scriptstyle{z}$] (10) at (5,1) {};
    \node[vert, label=right:$\scriptstyle{010101101010110}$] (11) at (6,2) {};

    \node[vert, label=below:$\scriptstyle{010101011011010}$] (12) at (6,0) {};
    \node[vert, label=right:$\scriptstyle{010101011010110}$] (13) at (7,1) {};

    \node[vert, label=below:$\scriptstyle{010101010110110}$] (14) at (8,0) {};
    
    \draw (0) -- (1) -- (2) -- (3) -- (4) -- (8) -- (11) -- (13) -- (14);
    \draw (5) -- (6) -- (7) -- (10) -- (12);
    \draw (1) -- (5) -- (6) -- (9) -- (10) -- (12) -- (13);
    \draw (2) -- (6);
    \draw (10) -- (11);
    \draw (3) -- (7) -- (8);

    \node (xyz) at (4,-1) {$\scriptstyle{x=010110101101010,\; y=010110101011010, \; z=010101101011010}$};

    \node (l1) at (4,-2) {$\pad{15}{8}$};

    \node[vert, label=above:$\scriptstyle{011011011011010}$] (15) at (3,6) {};
    \node[vert, label=below:$\scriptstyle{011011011010110}$] (16) at (4.2,6) {};
    \node[vert, label=above:$\scriptstyle{01101010110110}$] (17) at (5.4,6) {};
    \node[vert, label=below:$\scriptstyle{011010110110110}$] (18) at (6.8,6) {};
    \node[vert, label=above:$\scriptstyle{010110110110110}$] (19) at (8,6) {};

    \draw (15) -- (16) -- (17) -- (18) -- (19);
    
    \node (l2) at (6,7) {$\pad{15}{9}$};

    \node[vert, label=above:$\scriptstyle{010101010101010}$] (20) at (0,6) {};
    
    \node (l3) at (0,7) {$\pad{15}{7}$};
    
      \end{tikzpicture}
    \caption{All weighted Padovan graphs for $n = 15$.}
    \label{fig:padovan-15}
\end{figure}

\begin{definition}
    If $p$ and $q$ are nonnegative integers, then the graph $A_{p,q}$ is defined as follows. The vertex set of $A_{p,q}$ consists of all words of length $p+q$ over the alphabet $\{a,b\}$ which contain $p$ letters $a$ and $q$ letters $b$. Two vertices (alias words) are adjacent if one can be obtained from the other by changing a subword $ab$ to $ba$. 
\end{definition}

A \emph{weak partition} of $n \geq 0$ is a sequence of integers $\lambda = (\lambda_1, \ldots, \lambda_k)$ such that $\sum_{i = 1}^k \lambda_i = n$ and $\lambda_1 \geq \cdots \geq \lambda_k \geq 0$.  Terms $\lambda_1, \ldots, \lambda_k$ of $\lambda$ are called \emph{parts}, and $k$ is the number of parts of $\lambda$. Note that this is different from the usual definition of partition where only non-zero terms are considered to be parts of the weak partition. Alternatively, if $\lambda$ has $\alpha_i$ parts of size $i$, then it can be written as $\langle 0^{\alpha_0}, 1^{\alpha_1}, 2^{\alpha_2}, \ldots \rangle$. The number of all weak partitions of $n$ into $k$ parts with the largest part at most $j$ is denoted by $p(j,k,n)$. More on (weak) partitions can be found for example in~\cite{stanley}.

\begin{definition}
    If $p$ and $q$ are nonnegative integers, then the graph $\Pi_{p,q}$ is defined as follows. The vertex set of $\Pi_{p,q}$ consists of all weak partitions of $0, \ldots, pq$ into $q$ parts with the largest part of size at most $p$. (Recall that parts can be of size $0$ as well.) Two vertices (alias weak partitions) are adjacent if one can be obtained from the other by adding $1$ to one of the parts.
\end{definition}

\begin{theorem}
\label{thm:isomorphisms}
If  $n \geq 1$ and $\left \lfloor \frac{n}{2} \right \rfloor \leq k \leq \left \lfloor \frac{2n-2}{3} \right \rfloor$, then 
$$\pad{n}{k} \cong A_{2n-3k-2, 2k-n+1} \cong \Pi_{2n-3k-2, 2k-n+1}\,.$$
\end{theorem}

\proof
We define $\alpha: V(\pad{n}{k}) \rightarrow V(A_{2n-3k-2, 2k-n+1})$ as follows. Let $u\in V(\pad{n}{k})$. Then $\alpha(u)$ is a obtained from $u$ by replacing from left to right each $011$ by $b$, each $01$ by $a$, and removing the ending $0$. If $k_a$ and $k_b$ are the respective numbers of $a$s and $b$s in $\alpha(u)$, then the number of $1$s in $u$ is $k_a + 2 k_b$, that is, $k = k_a + 2 k_b$. Moreover, $2 k_a + 3 k_b = n-1$. From these two equations we obtain that $k_a = 2n - 3k - 2$ and $k_b = 2k - n +1$ which implies that $\alpha$ maps vertices of $\pad{n}{k}$ to vertices of $A_{2n-3k-2, 2k-n+1}$. Moreover, it is straightforward to check that $\alpha$ is a bijection. 

Let $uv\in E(\pad{n}{k})$. Then we may assume without loss of generality that $u = \ldots 01\ldots $ and $v = \ldots 10\ldots $, where $``\ldots"$ means that $u$ and $v$ coincide in all the other positions. By the definition of Padovan words we next infer that actually $u = \ldots 010110\ldots $ and $v = \ldots 011010\ldots $. This in turn implies that $\alpha(v)$ is obtained from $\alpha(u)$ by changing exactly one subword $ab$ to $ba$. So $\alpha$ maps edges to edges. Moreover, by the same argument as above we also see that $\alpha$ maps the vertices of $N_{\pad{n}{k}}(u)$ to the vertices of $N_{A_{2n-3k-2, 2k-n+1}}(\alpha(u))$. We may conclude that $\alpha$  is an isomorphism. 

Let $p = 2n - 3k - 2$ and $q = 2k - n +1$. We will show that $A_{p,q} \cong \Pi_{p,q}$. Let $\beta \colon V(A_{p,q}) \to V(\Pi_{p,q })$ be as follows. For $u \in V(A_{p,q})$, let $1 \leq i_1 \leq \cdots \leq i_q \leq p+q$ be positions of $b$s in $u$. Then $\beta(u)$ is the weak partition $((p+1) - i_1, \ldots, (p+q) - i_q)$. Since $u$ contains $q$ $b$s and is of length $p+q$, it clearly holds that $m \leq i_m \leq p+m$, thus all parts are between $0$ and $p$, and the sum of all parts is at most $pq$. Thus $\beta$ is well-defined and it is easy to see that it is a bijection.

Let $uv \in E(A_{p,q})$. Then we may assume without loss of generality that $u = \ldots ab \ldots $ and $v = \ldots ba\ldots $, where again $``\ldots"$ means that $u$ and $v$ coincide in all the other positions. This means that $v$ is obtained from $u$ by subtracting one from the position of one $b$. So $\beta(v)$ is obtained from $\beta(u)$ by adding $1$ to one part of the weak partition. Thus $\beta$ maps edges to edges. Moreover, by the same argument as above we also see that $\beta$ maps the vertices of $N_{A_{p,q}}(u)$ to the vertices of $N_{\Pi_{p,q}}(\beta(u))$, so $\beta$ is an isomorphism. 
\qed

\section{Order, size, degree}
\label{sec:order-size-degree}

In this section we determine the order, the size, and the degree of weighted Padovan graphs. Before that, their fundamental decomposition is described. 

The vertices of $\pad{n}{k}$ can be partitioned into those starting with $010$ and those starting with $011$. Let $X$ and $Y$ be the respective sets of vertices. Then $\pad{n}{k}[X] \cong \pad{n-2}{k-1}$, where $\pad{n}{k}[X]$ denotes the subgraph of $\pad{n}{k}$ induced by $X$. Moreover, since the vertices from $Y$ start by $0110$ we also see that $\pad{n}{k}[Y] \cong \pad{n-3}{k-2}$. Now, a vertex $u\in X$ has a neighbor in $Y$ if any only if  $u=010110\ldots$ in which case its neighbor from $Y$ is $011010\ldots$. Thus the vertices in $X$ that have an edge to $Y$ (and vice versa) induce $\pad{n-5}{k-3}$. The structure of $\pad{n}{k}$ as just described will be called the {\em fundamental decomposition} of $\pad{n}{k}$ and shortly denoted as 
\begin{equation*}
\pad{n}{k} = \mathbf{01}\pad{n-2}{k-1} + \mathbf{011}\pad{n-3}{k-2}\,.
\end{equation*}
For an example see Fig.~\ref{fig:pad-n-k},
where the fundamental decomposition $\pad{18}{10} = \mathbf{01}\pad{16}{9} + \mathbf{011}\pad{15}{8}$ is illustrated. To make the drawing clearer, the vertices are labeled via the isomorphism $\pad{18}{10} \cong A_{4,3}$. For example, this isomorphism maps the vertex $01\,011\,011\,011\,01\,01\,01\,0$ of $\pad{18}{10}$ to the vertex $abbbaaa$ of $A_{4,3}$. We thus have 
$$A_{4,3} \cong \pad{18}{10} = \mathbf{01}\pad{16}{9} + \mathbf{011}\pad{15}{8} \cong \mathbf{a}A_{3,3} + \mathbf{b}A_{4,2}\,.$$

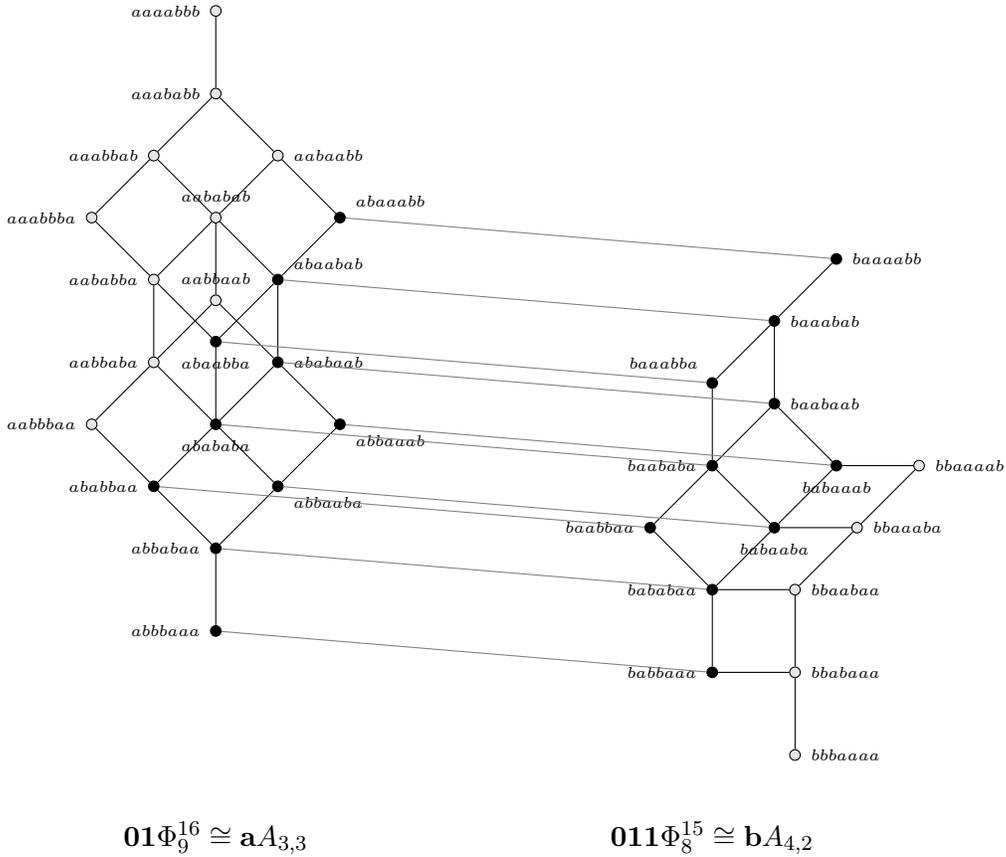
\begin{figure}[ht!]
    \centering
    \begin{tikzpicture}
    [scale=1.1,
    vert/.style={circle, draw, fill=black!10, inner sep=0pt, minimum width=4pt}, 
    double/.style={circle, draw, fill=black, inner sep=0pt, minimum width=4pt}, 
    central/.style={circle, draw, fill=black, inner sep=0pt, minimum width=4pt},
    ]

    \node[double, label=left:$\scriptscriptstyle{abbbaaa}$] (0) at (0,0) {};
    \node[double, label=left:$\scriptscriptstyle{abbabaa}$] (1) at (0,1) {};
    \node[double, label=left:$\scriptscriptstyle{ababbaa}$] (2) at (-0.75,1.75) {};
    \node[double, label=-5:$\scriptscriptstyle{abbaaba}$] (3) at (0.75,1.75) {};
    \node[vert, label=left:$\scriptscriptstyle{aabbbaa}$] (4) at (-1.5,2.5) {};
    \node[double, label=below:$\scriptscriptstyle{abababa}$] (5) at (0,2.5) {};
    \node[double, label=-5:$\scriptscriptstyle{abbaaab}$] (6) at (1.5,2.5) {};
    \node[vert, label=left:$\scriptscriptstyle{aabbaba}$] (7) at (-0.75,3.25) {};
    \node[double, label=right:$\scriptscriptstyle{ababaab}$] (8) at (0.75,3.25) {};
    \node[double, label=below:$\scriptscriptstyle{abaabba}$] (9) at (0,3.5) {};
    \node[vert, label=above:$\scriptscriptstyle{aabbaab}$] (10) at (0,4) {};
    \node[vert, label=left:$\scriptscriptstyle{aababba}$] (11) at (-0.75,4.25) {};
    \node[double, label={5:$\scriptscriptstyle{abaabab}$}] (12) at (0.75,4.25) {};
    \node[vert, label=left:$\scriptscriptstyle{aaabbba}$] (13) at (-1.5,5) {};
    \node[vert, label=above:$\scriptscriptstyle{aababab}$] (14) at (0,5) {};
    \node[double, label={5:$\scriptscriptstyle{abaaabb}$}] (15) at (1.5,5) {};
    \node[vert, label=left:$\scriptscriptstyle{aaabbab}$] (16) at (-0.75,5.75) {};
    \node[vert, label=right:$\scriptscriptstyle{aabaabb}$] (17) at (0.75,5.75) {};
    \node[vert, label=left:$\scriptscriptstyle{aaababb}$] (18) at (0,6.5) {};
    \node[vert, label=left:$\scriptscriptstyle{aaaabbb}$] (19) at (0,7.5) {};

    \draw (0) -- (1) -- (2) -- (4) -- (7) -- (10) -- (14) -- (16) -- (18) -- (19);
    \draw (1) -- (3) -- (6) -- (8) -- (12) -- (15) -- (17) -- (18);
    \draw (2) -- (5) -- (7) -- (11) -- (13) -- (16);
    \draw (3) -- (5) -- (8) -- (10);
    \draw (5) -- (9) -- (11) -- (14) -- (17);
    \draw (9) -- (12) -- (14);

    \node (p1) at (0,-2.5) {$\mathbf{01}\pad{16}{9} \cong \mathbf{a} A_{3,3}$};

    \node[double, label=left:$\scriptscriptstyle{babbaaa}$] (b0) at (6,-0.5) {};
    \node[double, label=left:$\scriptscriptstyle{bababaa}$] (b1) at (6,0.5) {};
    \node[double, label=left:$\scriptscriptstyle{baabbaa}$] (b2) at (5.25,1.25) {};
    \node[double, label=below:$\scriptscriptstyle{babaaba}$] (b3) at (6.75,1.25) {};
    \node[double, label=left:$\scriptscriptstyle{baababa}$] (b5) at (6,2) {};
    \node[double, label=below:$\scriptscriptstyle{babaaab}$] (b6) at (7.5,2) {};
    \node[double, label=right:$\scriptscriptstyle{baabaab}$] (b8) at (6.75,2.75) {};
    \node[double, label=above left:$\scriptscriptstyle{baaabba}$] (b9) at (6,3) {};
    \node[double, label=right:$\scriptscriptstyle{baaabab}$] (b12) at (6.75,3.75) {};
    \node[double, label=right:$\scriptscriptstyle{baaaabb}$] (b15) at (7.5,4.5) {};

    \node[vert, label=right:$\scriptscriptstyle{bbabaaa}$] (23) at (7,-0.5) {};
    \node[vert, label=right:$\scriptscriptstyle{bbbaaaa}$] (24) at (7,-1.5) {};
    \node[vert, label=right:$\scriptscriptstyle{bbaabaa}$] (22) at (7,0.5) {};
    \node[vert, label=right:$\scriptscriptstyle{bbaaaba}$] (21) at (7.75,1.25) {};
    \node[vert, label=right:$\scriptscriptstyle{bbaaaab}$] (20) at (8.5,2) {};
    
    \draw (b0) -- (b1) -- (b2);
    \draw (b1) -- (b3) -- (b6) -- (b8) -- (b12) -- (b15);
    \draw (b2) -- (b5);
    \draw (b3) -- (b5) -- (b8);
    \draw (b5) -- (b9);
    \draw (b9) -- (b12);

    \draw (24) -- (23) -- (22) -- (21) -- (20) -- (b6);
    \draw (23) -- (b0);
    \draw (22) -- (b1);
    \draw (21) -- (b3);

    \path[gray] (0) edge (b0);
    \path[gray] (1) edge (b1);
    \path[gray] (2) edge (b2);
    \path[gray] (3) edge (b3);
    \path[gray] (5) edge (b5);
    \path[gray] (6) edge (b6);
    \path[gray] (8) edge (b8);
    \path[gray] (9) edge (b9);
    \path[gray] (12) edge (b12);
    \path[gray] (15) edge (b15);

    \node (p2) at (6,-2.5) {$\mathbf{011}\pad{15}{8} \cong \mathbf{b} A_{4,2}$};

    \end{tikzpicture}
    \caption{The fundamental decomposition of $\pad{18}{10}$.}
    \label{fig:pad-n-k}
\end{figure}

\begin{proposition}
\label{prop:order-size}
If  $n \geq 1$ and $\left \lfloor \frac{n}{2} \right \rfloor \leq k \leq \left \lfloor \frac{2n-2}{3} \right \rfloor$, then 
$$|V(\pad{n}{k})| = \binom{n-k-1}{2n-3k-2}$$
and
$$|E(\pad{n}{k})| = |E(\pad{n-2}{k-1})| + |E(\pad{n-3}{k-2})| + |V(\pad{n-5}{k-3})|\,.$$
\end{proposition}

\proof
By Theorem~\ref{thm:isomorphisms} we have $\pad{n}{k} \cong A_{2n-3k-2, 2k-n+1}$. For the latter graph it is clear that $|V(A_{2n-3k-2, 2k-n+1})| = \binom{n-k-1}{2n-3k-2}$ because its vertices are words of length $(2n-3k-2) + (2k-n+1) = n-k-1$. 

By the fundamental decomposition, the vertices of the graph $\pad{n}{k}$ can be decomposed into $\mathbf{01} \pad{n-2}{k-1}$ and $\mathbf{011} \pad{n-3}{k-2}$. The only edges between these two sets are of the form $010110x \sim 011010x$, where $0x \in V(\pad{n-5}{k-3})$. From here the claimed formula for $|E(\pad{n}{k})|$ follows. 
\qed

\begin{theorem}
    \label{thm:number-edges}
    If $p \geq 1$ and $q \geq 1$, then $|E(A_{p,q})| = q \binom{p+q-1}{p-1}$, and if $pq=0$, then $|E(A_{p,q})| = 0$.
\end{theorem}

\begin{proof}
    If $p=0$ or $q=0$, $A_{p,q}$ contains at most one vertex, and has no edges.

    If $p=1$, then $V(A_{p,q}) = \{a b^{q}, bab^{q-1}, \ldots, b^q a\}$, and $A_{1,q}$ is isomorphic to a path on $q+1$ vertices and has $q = q \binom{1+q-1}{0}$ edges. Analogously, the formula holds if $q = 1$. 

    Suppose now that $p \geq 2$ and $q \geq 2$. It follows from Theorem \ref{thm:isomorphisms} that $A_{p,q} \cong \pad{2p+3q+1}{p+2q}$, and from Proposition \ref{prop:order-size} that $|E(A_{p,q})| = |E(A_{p-1,q})| + |E(A_{p,q-1})| + |V(A_{p-1,q-1})|$. As $p,q \geq 2$, $p-1, q-1 \geq 1$, we can use the induction hypothesis to obtain
    \begin{align*}
        |E(A_{p,q})| & = q \binom{p+q-2}{p-2} + (q-1) \binom{p+q-2}{p-1} + \binom{p+q-2}{p-1}\\
        & = \frac{(p+q-2)!}{(p-1)! (q-1)!} \left( p-1 + q-1 + 1 \right)\\
        & = \frac{(p+q-1)!}{(p-1)! (q-1)!} = q \binom{p+q-1}{p-1}. \qedhere
    \end{align*}
\end{proof}

\begin{corollary}
    \label{cor:edges-padovan}
    If $n \geq 1$ and $\left \lfloor \frac{n}{2} \right \rfloor \leq k \leq \left \lfloor \frac{2n-2}{3} \right \rfloor$, then $$
    |E(\pad{n}{k})| = \begin{cases} 
        0; & k \in \{ \frac{n-1}{2}, \frac{2n-2}{3} \},\\
        (2k-n+1) \binom{n-k-2}{2n-3k-3}; & \text{otherwise}.
    \end{cases}$$
\end{corollary}

Inspired by a comment in~\cite[Sequence A002457]{OEIS}  by Hans Haverman, we next give a combinatorial interpretation of the number of edges of weighted Padovan graphs which in turn provides an alternative proof of Theorem~\ref{thm:number-edges}.

\begin{theorem}
Let $p,q \geq 1$. Then the number of edges in $A_{p,q}$ is the same as the number of different words with $p-1$ letters $a$, $q-1$ letters $b$, and one letter $c$.
\end{theorem}

\begin{proof}
Recall that two vertices in $A_{p,q}$ are adjacent if one can be obtained from the other by changing one subword $ab$ to $ba$. Now every edge $x ab y \sim_{A_{p,q}} x ba y$ can be represented by $xcy$: the shared part in the beginning, then letter $c$, then shared part at the end (word representing the edge can, of course, also begin or end with $c$). It is easy to see that this representation yields a bijection between $E(A_{p,q})$ and the set of all different words with $p-1$ $a$s, $q-1$ $b$s and one $c$.
\end{proof}

\begin{proposition}
    \label{prop:degrees}
    If $p, q \geq 0$ and $p+q \geq 1$, then $$\delta(A_{p,q}) = \min\{1, pq\}$$ and $$\Delta(A_{p,q}) = \begin{cases}
        2 \min\{p,q\}; & p \neq q,\\
        2p-1; & p = q.
    \end{cases}$$
    Moreover, the number of vertices of degree $d$ in $A_{p,q}$ is equal to $$\begin{cases}
        \displaystyle{
        2 \left( \binom{p-1}{\frac{d-1}{2}} \binom{q-1}{\frac{d-1}{2}} \right);} & d \text{ is odd}, \\ \\ 
        \displaystyle{
        \binom{p-1}{\frac{d}{2} - 1} \binom{q-1}{\frac{d}{2}} + \binom{p-1}{\frac{d}{2}} \binom{q-1}{\frac{d}{2} - 1};} & d \text{ is even}.
    \end{cases}$$
\end{proposition}

\begin{proof}
    Let $v \in V(A_{p,q})$. Clearly, $\deg(v)$ is the sum of the number of appearances of the words $ab$ and $ba$ in $v$. Thus, to count the number of vertices of degree $d$, it suffices to count the number of words with $p$ $a$s, $q$ $b$s, and exactly $d$ subwords $ab$ and $ba$. Observe that occurrences of $ab$ and $ba$ in such a vertex alternate from left to right (and can overlap), so the vertex can be written in the form $a^{\alpha_1} b^{\beta_1} \cdots a^{\alpha_m} b^{\beta_m}$, where $\alpha_i, \beta_i \geq 1$ for all $i \in \{2,\ldots,m-1\}$, $\alpha_1, \beta_m \geq 0$, $\alpha_1 + \cdots + \alpha_m = p$, $\beta_1 + \cdots + \beta_m = q$ and $2m-1 - (1-\alpha_1) - (1-\beta_m) = d$ (which means that the number of nonempty blocks of $a$s and $b$s is equal to $d+1$).

    In the proof below, we will repeatedly use that the number of compositions of $n$ into $k$ parts, i.e.\ the number of integer solutions of $x_1 + \cdots + x_k = n$, $x_i \geq 1$ for all $i \in [k]$, is equal to $\binom{n-1}{k-1}$.

    If $\alpha_1, \beta_m \geq 1$, then $d = 2m-1$, so $d$ is odd, and $m = \frac{d+1}{2}$. Thus, by the comment above, there is $\binom{p-1}{\frac{d+1}{2}-1} \binom{q-1}{\frac{d+1}{2}-1}$ such vertices. Similarly we count the ones starting with $b$ and ending with $a$, that is, the ones with $\alpha_1 = \beta_m = 0$.

    If $\alpha_1 \geq 1$ and $\beta_m = 0$, then $d = 2m-2$, so $d$ is even, and $m = \frac{d}{2} + 1$. Thus there is $\binom{p-1}{\frac{d}{2}} \binom{q-1}{\frac{d}{2} - 1}$ such words. Similarly, we count the ones starting and ending with $b$, that is, the ones with $\alpha_1 = 0$ and $\beta_m \geq 1$.

    By definition, $\binom{n}{k} \geq 1$ if and only if $0 \leq k \leq n$, thus the formulas for minimum and maximum degree of $A_{p,q}$ follow from the determined number of vertices of degree $d$.
\end{proof}

\section{Metric properties}
\label{sec:metric}

In this section we determine several metric properties of weighted Padovan graphs. We begin with the distance function for which the partition representation $\Pi_{p,q}$ turns out to be the most convenient one. Using Theorem~\ref{thm:isomorphisms} the distance in $\pad{n}{k}$ and $A_{p,q}$ can be obtained as well.

\begin{proposition}
    \label{prop:distance-in-Pi}
    If $p,q \geq 0$ and $\lambda = (\lambda_1, \ldots, \lambda_q), \mu = (\mu_1, \ldots, \mu_q) \in V(\Pi_{p,q})$, then $d_{\Pi_{p,q}}(\lambda, \mu) = \sum_{i=1}^q |\lambda_i - \mu_i|$. 
\end{proposition}

\begin{proof}
    Recall that by definition, $\lambda_1 \geq \cdots \geq \lambda_q$ and $\mu_1 \geq \cdots \geq \mu_q$. Clearly, $d(\lambda, \mu) = \displaystyle \min \sum_{i=1}^q |\lambda_i - \mu_{\sigma(i)}|$, where the minimum is taken over all permutations $\sigma$ of $[q]$. The same $\sigma$ that minimizes this sum also minimizes $\sum_{i=1}^q (\lambda_i - \mu_{\sigma(i)})^2 = \sum_{i = 1}^q \lambda_i^2 + \sum_{i = 1}^q \mu_i^2 - 2 \sum_{i=1}^q \lambda_i \mu_{\sigma(i)}$. The $\sigma$ minimizing this expression is exactly the same as the one maximizing $\sum_{i=1}^q \lambda_i \mu_{\sigma(i)}$. By the rearrangement inequality from \cite[Theorem 368]{hardy-1952}, this sum is maximized if $\sigma(i)$ is such that $\mu_{\sigma(1)} \geq \cdots \geq \mu_{\sigma(q)}$, so when $\sigma$ is the identity.
\end{proof}

A graph $G$ is a {\em median graph} if for every triple of vertices $u,v,w$ of $G$ there exists a unique vertex $m(u,v,w)$, called {\em median}, which lies on a shortest $u,v$-path, on a shortest $u,w$-path, and on a shortest $v,w$-path. Recall that median graphs embed isometrically into hypercubes, hence by the subsequent theorem weighted Padovan graphs also have this property, that is, they are {\em partial cubes}. We refer to ~\cite[Chapters 4 and 6]{egecioglu-2023} for more on partial cubes and median graphs, and the relation of Fibonacci-like cubes to these classes of graphs. 

\begin{theorem}
\label{thm:median}
If $n\ge 1$ and $\left \lfloor \frac{n}{2} \right \rfloor \leq k \leq \left \lfloor \frac{2n-2}{3} \right \rfloor$, then $\pad{n}{k}$ is a median graph. 
\end{theorem}

\begin{proof}
Let $p=2n-3k-2$ and $q=2k-n+1$. Then by Theorem~\ref{thm:isomorphisms}, $\pad{n}{k} \cong \Pi_{p,q}$. Let $x = (\lambda_1,\ldots, \lambda_q)\in V(\Pi_{p,q})$. Then $0\le \lambda_i\le p$, $i\in [q]$. Set
$$\alpha_i(x) = \underbrace{0\ldots 0}_{p-\lambda_i} \underbrace{1\ldots 1}_{\lambda_i}, i\in [q],\quad {\rm and}\quad \alpha(x) = \alpha_1(x)\ldots \alpha_q(x)\,,$$
where $\alpha_1(x)\ldots \alpha_q(x)$ stands for the concatenation of the binary words $\alpha_1(x), \ldots, \alpha_q(x)$. The mapping $\alpha$ can thus be considered as
$$\alpha: V(\Pi_{p,q})\rightarrow V(Q_{pq})\,.$$
Let $x = (\lambda_1,\ldots, \lambda_q)$ and $y = (\mu_1,\ldots, \mu_q)$ be arbitrary vertices of $\Pi_{p,q}$. By Proposition~\ref{prop:distance-in-Pi} we have $d_{\Pi_{p,q}}(x,y) = \sum_{i=1}^q |\lambda_i - \mu_i|$. By the definition of $\alpha$, the words $\alpha_i(x)$ and $\alpha_i(y)$ differ in $|\lambda_i - \mu_i|$ positions. Hence  $\alpha(x)$ and $\alpha(y)$ differ in $\sum_{i=1}^q |\lambda_i - \mu_i|$ positions, that is, 
\begin{equation}
\label{eq:isometric-embedding}
d_{\Pi_{p,q}}(x,y) = d_{Q_{pq}}(\alpha(x), \alpha(y))\,.    
\end{equation}
Let $x = (\lambda_1,\ldots, \lambda_q)$, $y = (\mu_1,\ldots, \mu_q)$, and $z = (\tau_1,\ldots, \tau_q)$ be arbitrary vertices of $\Pi_{p,q}$. For each $i\in [q]$, let $\sigma_i$ be the median value of the set $\{\lambda_i, \mu_i, \tau_i\}$. Then in view of Proposition~\ref{prop:distance-in-Pi} and~\eqref{eq:isometric-embedding} we infer that $m(\alpha(x),\alpha(y),\alpha(z)) = \alpha(\sigma_1,\ldots, \sigma_q)$ is a median of $\alpha(x)$, $\alpha(y)$, and $\alpha(z)$. Moreover, $m(\alpha(x),\alpha(y),\alpha(z)) \in V(\alpha(\Pi_{p,q}))$. 

We have thus proved that $\alpha(\Pi_{p,q})$ is an isometric subgraph of $Q_{pq}$ such that with any three vertices of $\alpha(\Pi_{p,q})$, their median in $Q_{pq}$ is a vertex of $\alpha(\Pi_{p,q})$. Applying the theorem of Mulder~\cite{mulder-1978} (cf.\ also~\cite{klavzar-1999}) asserting that a graph $G$ is a median graph if and only if $G$ is a connected isometric subgraph of some $Q_n$ such that with any three vertices of $G$ their median in $Q_n$ is also a vertex of $G$, the argument is complete. 
\end{proof}

The proof of Theorem~\ref{thm:median} can be used to determine the diameter of weighted Padovan graphs.

\begin{proposition}
    If $n\ge 1$ and $\left \lfloor \frac{n}{2} \right \rfloor \leq k \leq \left \lfloor \frac{2n-2}{3} \right \rfloor$, then 
$$\diam(\pad{n}{k}) = (2n-3k-2)(2k-n+1)\,.$$ 
\end{proposition}

\begin{proof}
Let $p=2n-3k-2$ and $q=2k-n+1$. Since $\diam(Q_{pq}) = pq$ and the mapping $\alpha: V(\Pi_{p,q})\rightarrow V(Q_{pq})$ from the proof of Theorem~\ref{thm:median} is an isometric embedding, we get $\diam(\pad{n}{k}) \le pq$. On the other hand, $\alpha((0,\ldots, 0)) = 0^{pq}$  and $\alpha((p,\ldots, p)) = 1^{pq}$, so that 
\begin{align*}
\diam(\pad{n}{k}) & \ge d_{\Pi_{p,q}}((0,\ldots, 0), (p,\ldots, p))\\ 
& = d_{Q_{pq}}(\alpha((0,\ldots, 0)), \alpha((p,\ldots, p))) \\ 
& = d_{Q_{pq}}(0^{pq}, 1^{pq}) = pq = (2n-3k-2)(2k-n+1)\,. \qedhere
\end{align*}
\end{proof}

\section{The cube polynomial}
\label{sec:cube-poly}

The {\em cube polynomial} of a graph $G$ is denoted by $C\left( G,x\right)$, and is the generating function $C\left( G,x\right) =\sum_{n\geq 0}c_{n}\left(G\right) x^{n},$ where $c_{n}\left( G\right) $ counts the number of induced $n$-cubes in $G$. Clearly, $c_{0}\left( G\right) =|V(G)|$ and $c_{1}\left( G\right) =|E(G)|$. 

The cube polynomial was first studied in~\cite{bresar-2003}. Among the many subsequent investigations we point to its applicability in mathematical chemistry~\cite{berlic-2015, zhang-2013}, to its investigation on daisy cubes which as particular cases include Fibonacci cubes and Lucas cubes~\cite{klavzar-2019}, to the cube polynomial of tribonacci cubes~\cite{belbachir-2020}, and to an appealing relation between the cube polynomial and the clique polynomial~\cite{xie-2024}. For the cube polynomial of the weighted Padovan graphs we have:

\begin{theorem}
\label{thm:generating-cube-poly}
The generating function of the cube polynomial $C\left( \Phi
_{k}^{n},x\right) $ is%
\begin{equation*}
\sum_{n\geq 0}\sum_{k\geq 0}C\left( \Phi _{k}^{n},x\right) y^{n}z^{k}=\frac{y%
}{1-y^{2}z\left( 1+yz\left( 1+xy^{2}z\right) \right) }.
\end{equation*}
\end{theorem}

\begin{proof}
Let $Q_m$ be an induced subgraph of $\pad{n}{k}$. By the fundamental decomposition, exactly one of the following holds:
\begin{enumerate}
    \item $Q_m$ is an induced subgraph in $01\pad{n-2}{k-1} \cong \pad{n-2}{k-1}$,
    \item $Q_m$ is an induced subgraph in $011\pad{n-3}{k-2} \cong \pad{n-3}{k-2}$, or
    \item $Q_m = Q_{m-1} \cp K_2$, where the edges of $K_2$ correspond exactly to edges between copies of $01\pad{n-2}{k-1}$ and $011\pad{n-3}{k-2}$ in $\pad{n}{k}$, so $Q_{m-1}$ is an induced subgraph in $\pad{n-5}{k-3}$.
\end{enumerate}
Thus, if $n \geq 5$ and $k \geq 3$, then the cube polynomial of $\pad{n}{k}$ satisfies the following recursive relation:

$$
C\left( \Phi _{k}^{n},x\right) =C\left( \Phi _{k-1}^{n-2},x\right) +C\left(
\Phi _{k-2}^{n-3},x\right) +xC\left( \Phi _{k-3}^{n-5},x\right).
$$

Let the generating function of the cube polynomial be
\begin{equation*}
f\left( x,y,z\right) =\sum_{n\geq 0}\sum_{k\geq 0}C\left( \Phi
_{k}^{n},x\right) y^{n}z^{k}.
\end{equation*}
By using the recurrence relation and the
values 
\begin{itemize}
\item $\sum_{k\geq 0}C\left( \Phi _{k}^{0},x\right) y^{0}z^{k}=0$,  
\item $\sum_{k\geq 0}C\left( \Phi _{k}^{1},x\right) yz^{k}=y$,  
\item $\sum_{k\geq 0}C\left( \Phi _{k}^{2},x\right) y^{2}z^{k}=0$, 
\item $\sum_{k\geq 0}C\left( \Phi
_{k}^{3},x\right) y^{3}z^{k}=y^{3}z$,  \item $\sum_{k\geq 0}C\left( \Phi _{k}^{4},x\right) y^{4}z^{k}=y^{4}z^{2}$, 
\item $\sum_{k\geq 0}C\left( \Phi
_{k}^{n},x\right) y^{n}z^{k}=\sum_{k\geq 3}C\left( \Phi _{k}^{n},x\right)
y^{n}z^{k}+ y^{5}z^{2}$, $n\geq 5$,
\end{itemize}
we get 
\begin{align*}
&f\left( x,y,z\right) \left[ 1-y^{2}z-y^{3}z^{2}-xy^{5}z^{3}\right]  \\
&=\sum_{n\geq 0}\sum_{k\geq 0}C\left( \Phi _{k}^{n},x\right)
y^{n}z^{k}-\sum_{n\geq 2}\sum_{k\geq 1}C\left( \Phi _{k-1}^{n-2},x\right)
y^{n}z^{k} \\
&-\sum_{n\geq 3}\sum_{k\geq 2}C\left( \Phi _{k-2}^{n-3},x\right)
y^{n}z^{k}-x\sum_{n\geq 5}\sum_{k\geq 3}C\left( \Phi _{k-3}^{n-5},x\right)
y^{n}z^{k} \\
&=y+\sum_{n\geq 5}\sum_{k\geq 3}\left( C\left( \Phi _{k}^{n},x\right)
-C\left( \Phi _{k-1}^{n-2},x\right) -C\left( \Phi _{k-2}^{n-3},x\right)
-xC\left( \Phi _{k-3}^{n-5},x\right) \right) y^{n}z^{k} \\
&=y\,. \qedhere
\end{align*}
\end{proof}

By using the generating function of the cube polynomial, we obtain the cube polynomial itself as follows. 

\begin{theorem}
\label{thm:cube-poly}
If $n \geq 1$ and $\left \lfloor \frac{n}{2} \right \rfloor \leq k \leq \left \lfloor \frac{2n-2}{3} \right \rfloor$, then
\begin{equation*}
C\left( \Phi _{k}^{n},x\right) =\sum_{j\geq 0}\binom{n-k-j-1}{2k-n+1}\binom{
2k-n+1}{j}x^{j}.
\end{equation*} 
\end{theorem}
\begin{proof}
With Theorem~\ref{thm:generating-cube-poly} in hand, we can compute as follows:
\begin{align*}
f\left( x,y,z\right)  &=\frac{y}{1-y^{2}z\left( 1+yz\left( 1+xy^{2}z\right)
\right) } \\
&=y\sum_{n\geq 0}\sum_{k\geq 0}\binom{n}{k}\left( yz\right) ^{k}\left(
1+xy^{2}z\right) ^{k}\left( y^{2}z\right) ^{n} \\
&=\sum_{n\geq 0}\sum_{k\geq 0}\left( \sum_{j=0}^{k}\binom{n}{k}\binom{k}{j}
x^{j}\right) y^{2n+k+2j+1}z^{n+k+j} \\
&=\sum_{n\geq 0}\sum_{j\geq 0}\left( \sum_{k\geq j}\binom{n}{k}\binom{k}{j}
x^{j}\right) y^{2n+k+2j+1}z^{n+k+j} \\
&=\sum_{n\geq 0}\sum_{j\geq 0}\sum_{k\geq -n}\binom{n}{k-n-j}\binom{k-n-j}{j
}x^{j}y^{n+k+j+1}z^{k} \\
&=\sum_{n\geq 0}\sum_{j\geq 0}\sum_{k\geq n+j}\binom{n}{k-n-j}\binom{k-n-j}{
j}x^{j}y^{n+k+j+1}z^{k} \\
&=\sum_{n\geq 0}\sum_{k\geq j}\sum_{n\geq k-j}\binom{n}{k-n-j}\binom{k-n-j}{
j}x^{j}y^{n+k+j+1}z^{k} \\
&=\sum_{j\geq 0}\sum_{k\geq j}\sum_{n\geq -2j-1}\binom{n-k-j-1}{2k-n+1}
\binom{2k-n+1}{j}x^{j}y^{n}z^{k} \\
&=\sum_{n\geq 0}\sum_{k\geq 0}\left( \sum_{j\geq 0}\binom{n-k-j-1}{2k-n+1}
\binom{2k-n+1}{j}x^{j}\right) y^{n}z^{k} \,. \qedhere
\end{align*}
\end{proof}

The following result can be deduced from Theorem~\ref{thm:cube-poly} by determining the maximum degree of $C(\pad{n}{k},x)$ and the corresponding coefficient. However, we provide a different proof below since it illuminates the structure of the largest hypercubes contained in $\pad{n}{k}\cong A_{p,q}$.

\begin{proposition}
    \label{prop:largest-cube}
    The largest hypercube contained in $A_{p,q}$ as an induced subgraph is of size $\min\{p,q\}$, and the number of such hypercubes is $\max\{\binom{p}{q}, \binom{q}{p}\}$.
\end{proposition}

\begin{proof}
    Let $m = \min\{p,q\}$. Let $(ab)^m x \in V(A_{p,q})$. Vertices $\{ t_1\ldots t_m x \colon t_i \in \{ab, ba\}\}$ clearly induce $Q_m$ in $A_{p,q}$ (remove $x$ and replace each $ab$ with $0$ and each $ba$ with $1$).

    Suppose that $A_{p,q}$ contains $Q_{m+1}$ as an induced subgraph, and let $W$ be a vertex set of $Q_{m+1}$ in $A_{p,q}$. Observe that if $v \in V(Q_{m+1})$ and $x,y \in N(v)$, then $x,y$ have another common neighbor $z \neq v$. Without loss of generality assume that $m = q$. If $v = s aba t \in V(A_{p,q})$, then the vertices $sbaat$ and $saabt$ cannot have another common neighbor. Thus for every vertex $w \in W$ and for every $b$ in $w$, at most one neighbor of $w$ in $W$ is obtained by switching this $b$. Hence $\deg_{Q_{m+1}}(w)$ is at most the number of $b$s in $w$ which is equal to $q = m$. This is a contradiction since $Q_{m+1} \cong A_{p,q}[W]$ is $(m+1)$-regular.

    To count the number of $Q_m$s in $A_{p,q}$, again assume (without loss of generality) that $m = q$. After $q$ $ab$s are positioned into a word of length $p$, the whole word is determined (the remaining positions are filled with $a$s), and we obtain a word of length $p+q$ with $p$ $a$s and $q$ $b$s. This can be done in $\binom{p}{q}$ different ways, and each of them corresponds to a different induced hypercube $Q_m$.
\end{proof}

\section{Automorphisms}
\label{sec:automorphisms}

In this section we determine the automorphism group of weighted Padovan graphs. In view of Theorem~\ref{thm:isomorphisms}, this task is equivalent to the one of determining $\aut(\Pi_{p,q})$.

We first consider the trivial cases. 
If $\min\{p,q\} = 0$, then $\Pi_{p,q} \cong K_1$, so $\aut(\Pi_{p,q})$ is trivial. If $\min\{p,q\} = 1$, then $\Pi_{p,q}$ is isomorphic to the path on $\max\{p+1,q+1\}$ vertices, so $\aut(\Pi_{p,q}) = \mathbb{Z}_2$. Thus we assume that $\min\{p,q\} \geq 2$ in the following.

Let $G = \Pi_{p,q}$ through the whole section. The graph $G$ has exactly two leaves: $0 = (0,\ldots,0)$ and $pq = (p,\ldots,p)$. We say that a vertex $x \in V(G)$ is \emph{lonely} with respect to the leaf $\ell$ if it has exactly one neighbor $y \in V(G)$ such that $d(y,\ell) = d(x,\ell) - 1$. For short, we simply say that $x$ is lonely if it is lonely with respect to $0$. We say that vertices $y \in V(G)$ that are at distance $d\geq 0$ from the leaf $\ell$ belong to \emph{layer} $d$ with respect to $\ell$. When $\ell=0$, we simply say that they are in layer $d$.

\begin{lemma}
    \label{lem:lonely}
    If $x \in V(G)$ is lonely, then all nonzero parts in the weak partition of $x$ are of the same size.
\end{lemma}

\begin{proof}
    Suppose that $x$ has at least two nonzero parts of different sizes, so $x = (\lambda_1, \ldots, \lambda_q)$ and there exist $i,j \in [q]$, $i \neq j$, $\lambda_i > \lambda_j \geq 1$. Let $y$ be obtained from $x$ by changing $\lambda_i$ to $\lambda_i-1$, and let $z$ be obtained from $x$ by changing $\lambda_j$ to $\lambda_j-1$. Clearly $xy, xz \in E(G)$, and $y,z$ are in the layer $d(x,0) - 1$. Thus $x$ is not lonely.
\end{proof}

\begin{lemma}
    \label{lem:at-most-one-lonely-neighbor}
    If $w \in V(G)$ and $\ell$ is a leaf of $G$, then $w$ has at most one lonely (w.r.t.\ $\ell$) neighbor $x$ such that $d(x,\ell) = d(w,\ell) + 1$.
\end{lemma}

\begin{proof}
    Without loss of generality let $\ell=0$, and denote $d(w,0) = d$. Let $w = (\lambda_1, \ldots, \lambda_q)$. For $i\in [q]$, let $\mu_i$ be obtained from $\lambda$ by changing $\lambda_i$ to $\lambda_i+1$. Clearly, the neighbors of $w$ in layer $d+1$ are contained in the set $\{\mu_i \colon i \in [q]\}$. Suppose that some $\mu_i \in N(w)$ is lonely. Then $\lambda_j = \lambda_i+1$ for all $j \in [q]\setminus\{i\}$ by Lemma~\ref{lem:lonely}. But then, since $\min\{p,q\} \geq 2$, no other $\mu_i$ can be lonely since $\lambda_i$ would be smaller than the other part(s). Hence, $w$ has at most one lonely neighbor in layer $d+1$.
\end{proof}

\begin{lemma}
    \label{lem:at-most-one-common-neighbor}
    Let $y,z \in V(G)$, $y \neq z$, such that $d(y,\ell) = d(z,\ell)=d$ where $\ell$ is a leaf of $G$. Then $y$ and $z$ have at most one common neighbor in the layer $d+1$.
\end{lemma}

\begin{proof}
    Without loss of generality assume that $\ell=0$ and suppose that $y$ and $z$ have at least one common neighbor, $x$, such that $d(x,0)=d+1$. Since $x$ is not lonely, not all parts of $x$ are of the same size by Lemma~\ref{lem:lonely}. Let $y = (\lambda_1, \ldots, \lambda_q)$ and $z = (\mu_1, \ldots, \mu_q)$. Thus there exist $i \neq j$ such that $x$ is obtained from $y$ by changing $\lambda_i$ to $\lambda_i+1$, which is equal to obtaining $x$ from $z$ by changing $\mu_j$ to $\mu_j+1$. This means that $\lambda_k = \mu_k$ for all $k \in [q] \setminus \{i,j\}$, $\lambda_i+1=\mu_i$ and $\lambda_j = \mu_j+1$.

    Let $x'$ be another neighbor of $y$ in layer $d+1$. Then $x'$ is obtained from $y$ by changing $\lambda_m$ to $\lambda_m+1$ for some $m \in [q] \setminus \{i\}$. But then as $\lambda_m + 1 \neq \mu_m$, $x'$ is not a common neighbor of $y$ and $z$.
\end{proof}

Let $G(d,\ell) = \{ x \in V(G) \colon d(x,\ell) \leq d \}$ where $\ell$ is a leaf in $G$.

\begin{lemma}
    \label{lem:next-layer}
    Let $d \geq 2$ and let $\varphi \colon G(d,0) \to V(G)$. Suppose that $\varphi$ is an automorphism between $G(d,0)$ and $\varphi(G(d,0))$. Then there is at most one possibility of extending $\varphi$ to $\psi \colon G(d+1,0) \to V(G)$ such that $\psi|_{G(d,0)} = \varphi$ and $\psi$ is an automorphism between $G(d+1,0)$ and $\psi(G(d+1,0))$.
\end{lemma}

\begin{proof}
    Let $x \in G(d+1,0) \setminus G(d,0)$. Observe that since $\varphi$ is an automorphism on its image, $\varphi(0)$ is also a leaf.

        \medskip \noindent \textbf{Case 1}: $x$ is not lonely.\\
        Then there exist vertices $y \neq z$ in layer $d$ that are both neighbors of $x$. By Lemma~\ref{lem:at-most-one-common-neighbor} $x$ is the unique common neighbor of $y$ and $z$. Vertices $\varphi(y)$ and $\varphi(z)$ are in layer $d$ with respect to $\varphi(0)$ and by Lemma~\ref{lem:at-most-one-common-neighbor} they have at most one common neighbor in layer $d+1$ with respect to $\varphi(0)$, say $x'$. Since $\varphi$ is an automorphism on its image, $\psi(x)$ can only be $x'$ (if it exists).

        \medskip \noindent \textbf{Case 2}: $x$ is lonely.\\
        Then $x$ has exactly one neighbor $w$ in layer $d$. By Lemma~\ref{lem:at-most-one-lonely-neighbor} $x$ is the only lonely neighbor of $w$. The vertex $\varphi(w)$ is in layer $d$ with respect to $\varphi(0)$, and by Lemma~\ref{lem:at-most-one-lonely-neighbor} it has at most one lonely neighbor, say $x'$. Thus $\psi(x)$ can only be $x'$ (if it exists).
\end{proof}

\begin{lemma}
    \label{lem:no-extension}
    Let $p \neq q$ and let $\varphi \colon G(2,0) \to V(G)$ map as follows:
    \begin{align*}
        \langle 0^q \rangle & \mapsto \langle 0^q \rangle,\\
        \langle 1, 0^{q-1} \rangle & \mapsto \langle 1, 0^{q-1} \rangle,\\
        \langle 2, 0^{q-1} \rangle & \mapsto \langle 1^2, 0^{q-2} \rangle,\\
        \langle 1^2, 0^{q-2} \rangle & \mapsto \langle 2, 0^{q-1} \rangle.
    \end{align*} Then $\varphi$ cannot be extended to an automorphism of $G$.
\end{lemma}

\begin{proof}
    Let $p < q$. Suppose that it is possible to extend $\varphi$ to $\psi \colon G(p,0) \to V(G)$ such that $\psi$ is an automorphism on its image. Then we show that $\psi(\langle p, 0^{q-1} \rangle) = \langle 1^p, 0^{q-p} \rangle$, $\psi(\langle 1^p, 0^{q-p} \rangle) = \langle p, 0^{q-1} \rangle$. This is clearly true for $p=2$. Consider $p \geq 3$ and assume that it is true for $p-1$. Neighbors of $\langle p-1, 0^{q-1} \rangle$ in layer $p$ are $\langle p, 0^{q-1} \rangle$ (which is lonely) and $\langle p-1,1, 0^{q-2} \rangle$. Neighbors of $\langle 1^{p-1}, 0^{q-p+1} \rangle$ in layer $p$ are $\langle 1^p, 0^{q-p} \rangle$ (which is lonely) and $\langle 2, 1^{p-2}, 0^{q-p+1} \rangle$. Since $\psi$ should be an automorphism, lonely neighbors have to be mapped into lonely neighbors, thus $\psi(\langle p, 0^{q-1} \rangle) = \langle 1^p, 0^{q-p} \rangle$, $\psi(\langle 1^p, 0^{q-p} \rangle) = \langle p, 0^{q-1} \rangle$. Note that $\psi$ is unique by Lemma~\ref{lem:next-layer}.
    
    Observe that $G(p+1,0)\setminus G(p,0)$ contains all weak partitions of $p+1$ except from $\langle p+1, 0^{q-1} \rangle$. Thus $\langle p, 0^{q-1} \rangle \in G(p,0)$ has only one neighbor in layer $p+1$: $\langle p,1, 0^{q-2} \rangle$. However, $\psi(\langle p, 0^{q-1} \rangle) = \langle 1^p, 0^{q-p} \rangle$ has two neighbors in layer $p+1$: $\langle 1^{p+1}, 0^{q-p-1} \rangle$ and $\langle 2, 1^{p-1}, 0^{q-p} \rangle$. Thus, $\psi$ cannot be extended into an automorphism on $G(p+1,0)$.

    An analogous argument settles the case when $p > q$.
\end{proof}

\begin{theorem}
    \label{thm:automorphisms}
    Let $\min\{p,q\} \geq 2$. If $p \neq q$, then $\aut(G) = \mathbb{Z}_2$, and if $p = q$, then $\aut(G) = \mathbb{Z}_2 \times \mathbb{Z}_2$.
\end{theorem}

\begin{proof}
    Since $G$ has exactly two leaves, $0 = \langle 0^q \rangle$ and $pq = \langle p^q \rangle$, every automorphism of $G$ either maps $0 \mapsto 0$, $pq \mapsto pq$, or $0 \mapsto pq$, $pq \mapsto 0$. This uniquely determines how the unique neighbors $1 = \langle 1, 0^{q-1} \rangle$ and $pq-1 = \langle p^{q-1},p-1 \rangle$ of the leaves are mapped. However, since $\min\{p,q\} \geq 2$, vertices $1$ and $pq-1$ each have two neighbors, which can again be mapped either identically or into each other. For illustration, see Figure~\ref{fig:automorphism}.

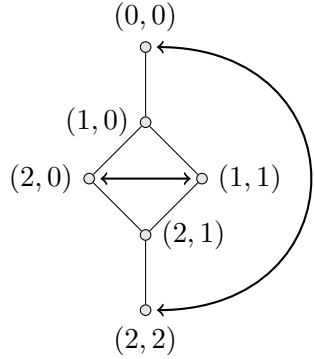
\begin{figure}[ht!]
    \centering
    \begin{tikzpicture}
    [scale=1.0,
    vert/.style={circle, draw, fill=black!10, inner sep=0pt, minimum width=4pt}, 
    double/.style={circle, draw, fill=black, inner sep=0pt, minimum width=4pt}, 
    central/.style={circle, draw, fill=black, inner sep=0pt, minimum width=4pt},
    ]

    \node[vert, label=below:{$(2,2)$}] (0) at (0,0) {};
    \node[vert, label=right:{$(2,1)$}] (1) at (0,1) {};
    \node[vert, label=left:{$(2,0)$}] (2) at (-0.75,1.75) {};
    \node[vert, label=right:{$(1,1)$}] (3) at (0.75,1.75) {};
    \node[vert, label=left:{$(1,0)$}] (5) at (0,2.5) {};
    \node[vert, label=above:{$(0,0)$}] (9) at (0,3.5) {};
    
    \draw (0) -- (1) -- (2);
    \draw (1) -- (3);
    \draw (2) -- (5) -- (9);
    \draw (3) -- (5);   

    \draw[<->, thick] (-0.6,1.75) -- (0.6,1.75);
    \draw[<->, thick] (0.15,0) to[bend right=90, looseness=2] (0.15,3.5);
    \end{tikzpicture}
    \caption{The smallest weighted Padovan graph ($\Pi_{2,2}$) with the automorphism group isomorphic to the Klein-four group $\mathbb{Z}_2 \times \mathbb{Z}_2$.}
    \label{fig:automorphism}
\end{figure}

    Let $\varphi \colon V(G) \to V(G)$ be an automorphism. By the above, there are only $2 \cdot 2 = 4$ possibilities of how $\varphi$ maps $N[1]$. By inductively applying Lemma~\ref{lem:next-layer}, $\varphi|_{N[1]}$ determines the whole $\varphi$ uniquely. Thus $\aut(G) \subseteq \mathbb{Z}_2 \times \mathbb{Z}_2$. 

    However, if $p \neq q$ and $\varphi(N[1]) = N[1]$, then Lemma~\ref{lem:no-extension} shows that $\varphi|_{N[1]} = id|_{N[1]}$. Thus $\aut(G) \neq \mathbb{Z}_2 \times \mathbb{Z}_2$, and by the Lagrange's theorem, $|\aut(G)| \in \{1,2\}$.

    Let $\tau \colon V(G) \to V(G)$ be defined so that $\tau((\lambda_1, \ldots, \lambda_q)) = (p-\lambda_q, \ldots, p-\lambda_1)$. Clearly, $\tau$ is well-defined. Let $xy \in E(G)$. So $x = (\lambda_1, \ldots, \lambda_q)$ and $y$ is obtained from $x$ by changing $\lambda_i$ to $\lambda_i+1$ for some $i \in [q]$. Thus $\tau(x) = (p-\lambda_q, \ldots, p - \lambda_i, \ldots, p-\lambda_1)$ and $\tau(y) = (p-\lambda_q, \ldots, p-\lambda_i-1, \ldots, p-\lambda_1)$, so $\tau(x) \tau(y) \in E(G)$. Similarly we obtain that $\tau(N(x)) = N(\tau(x))$. Hence, $\tau$ is an automorphism. Thus if $p \neq q$, $\aut(G) = \mathbb{Z}_2 = \{id, \tau\}$.

    If $p = q$, let $\rho \colon V(G) \to V(G)$ be defined as $\rho(\lambda) = \lambda'$. Recall that the \emph{conjugate} $\lambda'$ of a weak partition $\lambda = (\lambda_1, \ldots, \lambda_q)$ is obtained by interchanging the rows and columns of the Ferrers diagram of $\lambda$. In other words, the number of parts of $\lambda'$ equal to $i$ is $\lambda_i - \lambda_{i+1}$. Since $\lambda$ has $q$ parts which are all at most $p$, $\lambda'$ has at most $p$ parts which are all at most $q$. Finally, we add a sufficient number of parts of size 0 to the end of $\lambda'$ so it has exactly $p$ parts. Since $p=q$, $\rho$ is well-defined. Let $xy \in E(G)$. So $x = (\lambda_1, \ldots, \lambda_q)$ and $y$ is obtained from $x$ by changing $\lambda_i$ to $\lambda_i+1$ for some $i \in [q]$. Consider now $\rho(x)$ and $\rho(y)$. The number of parts that equal $j$ is equal in both $\rho(x)$ and $\rho(y)$ for all $j \in [q] \setminus \{ i-1, i\}$, while the number of parts of size $i-1$ and $i$ differ for $1$ between $\rho(x)$ and $\rho(y)$. Thus $\rho(x) \rho(y) \in E(G)$. Similarly, we see that $\rho(N(x)) = N(\rho(x))$, and so $\rho$ is an automorphism of $G$. Since $\rho(0) = 0$ and $\tau(0) = pq$, $\rho \neq \tau$. Thus $\aut(G) = \mathbb{Z}_2 \times \mathbb{Z}_2 = \{id, \tau, \rho, \tau \circ \rho\}$.     
\end{proof}

\section*{Acknowledgements}

This work has been supported by T\"{U}B\.{I}TAK and the Slovenian Research Agency under grant numbers 122N184 and BI-TR/22-24-002, respectively. Vesna Ir\v{s}i\v{c}, Sandi Klav\v{z}ar, and Gregor Rus acknowledge the financial support from the Slovenian Research Agency (research core funding P1-0297 and projects N1-0218, N1-0285, N1-0355, Z1-50003). Vesna Ir\v{s}i\v{c} also acknowledges the financial support from the European Union (ERC, KARST, 101071836).


\begin{thebibliography}{99}

\bibitem{belbachir-2020}
H.~Belbachir, R.~Ould-Mohamed, 
Enumerative properties and cube polynomials of tribonacci cubes,
Discrete Math.\ 343 (2020) 111922.

\bibitem{berlic-2015}
M.~Berli\v{c}, N.~Tratnik, P.~\v{Z}igert Pleter\v{s}ek, 
Equivalence of Zhang-Zhang polynomial and cube polynomial for spherical benzenoid systems,
MATCH Commun.\ Math.\ Comput.\ Chem.\ 73 (2015) 443--456.

\bibitem{bresar-2003}
B.~Bre\v{s}ar, S.~Klav\v{z}ar, R.~\v{S}krekovski, 
The cube polynomial and its derivatives: the case of median graphs,
Electron.\ J.\ Combin.\ 10 (2003) 3.

\bibitem{doslic-2024}
T.~Do\v{s}li\'{c}, L.~Podrug,
Metallic cubes, 
Discrete Math.\ 347 (2024) 113851. 

\bibitem{egecioglu-irsic-2021} 
\"{O}.~E\u{g}ecio\u{g}lu, V.~Ir\v{s}i\v{c}, 
Fibonacci-run graphs {I}: {B}asic properties,
Discrete Appl.\ Math.\ 295 (2021) 70--84.

\bibitem{egecioglu-2023} 
\"{O}.~E\u{g}ecio\u{g}lu, S.~Klav\v{z}ar, M.~Mollard,
Fibonacci Cubes with Applications and Variations,
World Scientific, Singapore, 2023.

\bibitem{egecioglu-2021b}
\"O.~E\u gecio\u glu, E.~Sayg\i, Elif Z.~Sayg\i,
The {M}ostar index of {F}ibonacci and {L}ucas cubes,
Bull.\ Malays.\ Math.\ Sci.\ Soc.\ 44 (2021) 3677--3687.

\bibitem{hardy-1952}
G.\ H.\ Hardy, J.\ E.\ Littlewood, G.\ Pólya,
Inequalities, Reprint of the 1952 edition,
Cambridge University Press, Cambridge, 1988.

\bibitem{hsu-1993} 
W.-J.~Hsu, 
Fibonacci cubes---a new interconnection topology, 
IEEE Trans.\ Parallel Distr. Systems 4 (1993) 3--12.

\bibitem{irsic-2023} 
V.~Ir\v{s}i\v{c}, S.~Klav\v{z}ar, E.~Tan,
Generalized Pell graphs,
Turkish J.\ Math.\ 47 (2023) 1955--1973.

\bibitem{klavzar-2013} 
S.~Klav\v{z}ar, 
Structure of Fibonacci cubes: a survey, 
J. Comb.\ Optim.\ 25 (2013) 505--522.

\bibitem{klavzar-2019} 
S.~Klav\v{z}ar, M.~Mollard,
Daisy cubes and distance cube polynomial,
European J.\ Combin.\ 80 (2019) 214--223.

\bibitem{klavzar-1999} 
S.~Klav\v{z}ar, H.M.~Mulder, 
Median graphs: characterizations, location theory and related structures, 
J.\ Combin.\ Math.\ Combin.\ Comput.\ 30 (1999) 103--127.

\bibitem{klavzar-2023} 
S.~Klav\v{z}ar, E.~Tan, 
Edge general position sets in Fibonacci and Lucas cubes,
Bull.\ Malays.\ Math.\ Sci.\ Soc.\ 46 (2023) 120.

\bibitem{koshy-2014}
T.~Koshy, 
Pell and {P}ell-{L}ucas {N}umbers with {A}pplications,
Springer, New York, 2014. 

\bibitem{koshy-2019}
T.~Koshy, 
Fibonacci and {L}ucas {N}umbers with {A}pplications. {V}ol.\ 2, 
John Wiley \& Sons, Inc., Hoboken, NJ, 2019. 

\bibitem{lee-2023-1} 
G.~Lee, J.~Kim,
On the Padovan codes and the Padovan cubes, 
Symmetry 15 (2023) 266.

\bibitem{lee-2023-2} 
G.~Lee, J.~Kim,
On the cube polynomials of Padovan and Lucas-Padovan cubes,
Symmetry 15 (2023) 1389.

\bibitem{mollard-2021}
M.~Mollard, 
Edges in Fibonacci cubes, Lucas cubes and complements,
Bull.\ Malays.\ Math.\ Sci.\ Soc.\ 44 (2021) 4425--4437.

\bibitem{mulder-1978}
H.M.~Mulder, 
The structure of median graphs,
Discrete Math.\ 24 (1978) 197--204.

\bibitem{munarini-2001} 
E.~Munarini, C.~Perelli Cippo, N.~Zagaglia Salvi, 
On the Lucas cubes, 
Fibonacci Quart.\ 39 (2001) 12--21.

\bibitem{munarini-2019} 
E.~Munarini, 
Pell graphs, 
Discrete Math.\ 342 (2019) 2415--2428.

\bibitem{OEIS} 
The On-line Encyclopedia of Integers Sequences, 
The OEIS Foundation Inc., 
Published electronically at https://oeis.org. 

\bibitem{padovan} 
R.~Padovan, 
Dom Hans van der Laan and the plastic number, 
Architecture and Mathematics from Antiquity to the Future, 
Volume II, K.~Williams and M.J.~Ostwald (eds.), Chapter 74, pp.\ 407--419, Springer, 2015.

\bibitem{stanley}
R.P.~Stanley,  
Enumerative Combinatorics: Volume 1, 
Cambridge University Press, Cambridge, 2011. 

\bibitem{xie-2024}
Y.-T.~Xie, Y.-D.~Feng, S.-J.~Xu, 
A relation between the cube polynomials of partial cubes and the clique polynomials of their crossing graphs,
J.\ Graph Theory 106 (2024) 907--922.

\bibitem{wei-2022} 
J.~Wei, Y.~Yang, 
Fibonacci and Lucas $p$-cubes, 
Discrete Appl.\ Math.\ 322 (2022) 365--383.

\bibitem{zhang-2013}
H.~Zhang, W.C.~Shiu, P.K.~Sun, 
A relation between Clar covering polynomial and cube polynomial,
MATCH Commun.\ Math.\ Comput.\ Chem.\ 70 (2013) 477--492.

\end{thebibliography}
\end{document}